\documentclass{amsart}

\usepackage{amsfonts,amssymb,amsmath,enumerate,verbatim,mathtools,tikz}
\usetikzlibrary{matrix}
\usetikzlibrary{arrows}

\usepackage[all]{xy}
\SelectTips{cm}{}

\newcommand\BZ{{\mathbb Z}}

\newcommand{\col}{\colon}

\newcommand\depth{\operatorname{depth}}
\newcommand\codim{\operatorname{codim}}

\newcommand\ext{\operatorname{Ext}}
\newcommand\tor{\operatorname{Tor}}

\newcommand{\Extb}[4]{\overline{\operatorname{Ext}}{\vphantom E}^{#1}_{#2}(#3,#4)}
\newcommand{\Extv}[4]{\widehat{\operatorname{Ext}}{\vphantom E}^{#1}_{#2}(#3,#4)}

\newcommand\fm{{\mathfrak m}}
\newcommand\fn{{\mathfrak n}}

\newcommand\fg{{\mathfrak g}}

\newcommand\me{{\mathcal{E}}}

\newcommand\ms{{\mathcal{S}}}
\newcommand\mb{{\mathcal{B}}}
\newcommand\mt{{\mathcal{T}}}

\newcommand\ma{{\mathcal{A}}}
\newcommand\mm{{\mathcal{M}}}

\newcommand{\End}{\operatorname{End}}
\newcommand\Hom{\operatorname{Hom}}
\newcommand\Der{\operatorname{Der}}

\newcommand{\hh}[1]{\operatorname{H}(#1)}
\newcommand{\HH}[2]{\operatorname{H}^{#1}(#2)}

\newcommand{\lra}{\longrightarrow}
\newcommand{\ov}{\overline}
\newcommand{\wh}{\widehat}

\newcommand{\xra}{\xrightarrow}

\newtheorem{theorem}{Theorem}[section]
\newtheorem*{theorem*}{Theorem}
\newtheorem*{corollary*}{Corollary}
\newtheorem{proposition}[theorem]{Proposition}
\newtheorem{lemma}[theorem]{Lemma}
\newtheorem{corollary}[theorem]{Corollary}

\theoremstyle{definition}
\newtheorem{example}[theorem]{Example}
\newtheorem*{definition}{Definition}

\theoremstyle{remark}
\newtheorem{remark}[theorem]{Remark}

\newtheorem{chunk}[theorem]{}

\numberwithin{equation}{section}

\begin{document}
\title[A bimodule structure for bounded cohomology]{A bimodule structure for the bounded cohomology of commutative local rings}
\author{Luigi Ferraro}
\maketitle
\begin{abstract}
Stable cohomology is a generalization of Tate cohomology to associative rings, first defined by Pierre Vogel. For a commutative local ring $R$ with residue field $k$, stable cohomology modules $\wh\ext{\vphantom E}^{n}_R\;(k,k)$, defined for $n\in\mathbb{Z}$, have been studied by Avramov and Veliche. Stable cohomology carries a structure of $\mathbb{Z}$-graded $k$-algebra. One of the main goals of this paper is to prove that, for a class of Gorenstein rings, this algebra is a trivial extension of absolute cohomology $\ext_R(k,k)$ and a shift of $\Hom_k(\ext_R(k,k),k)$. We use this information to characterize the rings $R$ for which stable cohomology is graded-commutative. Stable cohomology is connected through an exact sequence to bounded cohomology. We use this connection to understand the algebra structure of $\wh\ext_R(k,k)$ by investigating the structure of bounded cohomology $\ov\ext_R(k,k)$ as a graded $\ext_R(k,k)$-bimodule.

\end{abstract}
\section*{Introduction}
Let $R$ be an associative ring and $(M,N)$ a pair of left $R$-modules, then stable cohomology associates to this pair, groups
\[
\wh\ext{\vphantom E}^{n}_R\;(M,N),\quad n\in\mathbb{Z}
\]
which are all zero if $M$ or $N$  has finite projective dimension. There is a canonical transformation $\iota:\ext_R\rightarrow\wh\ext_R$ of absolute cohomology to stable cohomology, which we use to study the relation between the multiplicative structures of these two theories.

Stable cohomology was introduced in an unpublished work by P. Vogel. The first appearance of stable cohomology in published form was in \cite{goichot}, where Goichot called it Tate-Vogel cohomology. It is a generalization of Tate cohomology for modules over finite group rings. In commutative algebra this cohomology was studied by Avramov and Veliche \cite{AV} under the name stable cohomology, for it brings out its relation to the stabilization of module categories.

We focus on local commutative Noetherian rings $(R,\fm,k)$ that are not regular, since in that case $\wh\ext{\vphantom E}^n_R(-,-)=0$ for every $n$. Under this hypothesis Martsinkovsky proved that $\iota:\ext_R(k,k)\rightarrow\wh\ext_R(k,k)$ is an injective map, see \cite{mart}. Our goal is to understand the algebra structure of $\wh\ext_R(k,k)$. To do this we determine the $\ext_R(k,k)$-bimodule structure of the cokernel of $\iota$, which is the bounded cohomology $\ov\ext_R(k,k)$, up to a shift. The left module structure of this cokernel was already studied in \cite{AV}. We describe the right module structure and use it to determine the structure of $\wh\ext_R(k,k)$ for Gorenstein rings for which $\iota$ is split as a map of $\ext_R(k,k)$-bimodules.

%We then direct our attention to the complete intersection case, in which we give a complete characterization of the algebra structure of $\wh\ext_R(k,k)$.

In Section 1 we recall the notions of DG Lie algebra and module, which will be used in Section 3 to construct a bimodule structure on bounded cohomology. In Section 2 we recall the construction of stable and bounded cohomology.

Section 3 is the core of the paper. In this section we define a $\ext_R(N,N)$-$\ext_R(k,k)$-bimodule structure on $\ext_R(k,R)\otimes_k\tor^R(k,N)$. To define this bimodule structure we make extensive use of the fact that the absolute cohomology algebra $\ext_R(k,k)$ is the universal enveloping algebra of a graded Lie algebra. The existence of such a Lie algebra is one of the many parallelisms between local algebra and rational homotopy theory. The main result of this section is that $\ext_R(k,R)\otimes_k\tor^R(k,N)$ is isomorphic as a bimodule to $\ov\ext_R(k,N)$. This answers a question raised by Avramov and Veliche in \cite{AV}.

In Section 4 we define a $\ext_R(k,k)$-$\ext_R(N,N)$-bimodule structure on the $k$-dual of $\tor^R(N,k)$, which makes it isomorphic to $\ext_R(N,k)$ as a bimodule. We then prove that this isomorphism restricts to the same isomorphism of $R^!$-$\ext_R(N,N)$-bimodules, where $R^!$ is the Koszul dual of $R$, that was defined by Iyengar and Herzog in \cite{hi}. We use this result in Section 5, together with the main result of Section 3, to prove that $\wh\ext_R(k,k)$ is isomorphic to a trivial extension algebra of $\ext_R(k,k)$ and a shift of its $k$-dual, provided that $R$ is a Gorenstein ring with $\mathrm{depth}\;\ext_R(k,k)\geq2$. For any non regular ring one has $\mathrm{depth}\;\ext_R(k,k)\geq1$, so the previous condition on the depth of the absolute cohomology algebra is not too restrictive. Avramov and Veliche have also proved, in \cite[Lemma 8.3]{AV}, that any complete intersection of codimension at least 2 satisfies $\mathrm{depth}\;\ext_R(k,k)\geq2$.

In Section 6 we use the structural information on $\wh\ext_R(k,k)$, obtained in Section 5, to characterize the rings $R$ for which the algebra $\wh\ext_R(k,k)$ is graded-commutative.

The homological properties of Golod rings are, in many respects, opposite to those of Gorenstein rings. In Section 5 a lot of information is gathered on the right module structure of bounded cohomology over Gorenstein rings. It is only natural to wonder what one can say about this structure when the ring is Golod. The simplest Golod rings are rings with $\fm^2=0$. This is why we dedicate the last section of the paper to studying the right multiplication of $\ext_R(k,k)$ on $\ov\ext_R(k,k)$ for rings with $\fm^2=0$.

\section{DG Lie Algebras and Modules}
Let $R$ be a commutative ring.
Let $\fg$ be a DG Lie algebra over $R$ with differential $\partial^\fg$, see \cite[Chapter 10]{infinite} for the definition. A DG $R$-module $M$ is a (right) DG Lie $\fg$-module if there exists a map
\[
M\otimes_R\fg\rightarrow M
\]
satisfying the following conditions, for $m\in M$ and $\theta,\xi\in\fg$, where we denote $m\otimes\theta$ by $m\cdot\theta$:\\
1) $\partial^M(m\cdot\theta)=\partial^M(m)\cdot\theta+(-1)^{|m|}m\cdot\partial^\fg(\theta)$, where $\partial^M$ is the differential of $M$,\\
2) $m\cdot[\theta,\xi]=(m\cdot\theta)\cdot\xi-(-1)^{|\theta||\xi|}(m\cdot\xi)\cdot\theta$,\\
3) $m\cdot\theta^{[2]}=(m\cdot\theta)\cdot\theta$, for $\theta\in\fg^{\mathrm{odd}}$.\\
\noindent
The definition of DG left $\fg$-module is similar.

If $M$ is a DG left $\fg$-module we can turn it into a DG right $\fg$-module in the following way
\[
m\cdot\theta:=-(-1)^{|\theta||m|}\theta\cdot m,\quad m\in M,\theta\in\fg,
\]
a routine computation shows that this is indeed an action.

If $M$ and $N$ are DG right $\fg$-modules then $M\otimes_R N$ is a DG right $\fg$-module with action
\[
(m\otimes n)\cdot x:=m\otimes(n\cdot x)+(-1)^{|x||n|}(m\cdot x)\otimes n,\quad m\in M,n\in N,x\in\fg,
\]
similarly for tensor product of left modules.

If $\fg$ is a graded Lie $k$-algebra with $k$ a field, we denote by $U\fg$ its universal enveloping algebra, see \cite[Chapter 10]{infinite} for the definition. A Lie $\fg$-module is just a $U\fg$-module.
\section{Stable and bounded cohomology}
In this section we recall the construction of stable cohomology. Let $R$ be a commutative ring, and let $L$ and $M$ be $R$-modules.
Choose projective resolutions $P$ and $Q$ of $L$ and $M$, respectively. Recall that a homomorphism
$P\to Q$ of degree $n$ is a family $\beta=(\beta_i)_{i\in\BZ}$ of $R$-linear maps $\beta_i\col P_i\to Q_{i+n}$;
that means an element of the $R$-module
\[
{\Hom}_{R}(P,Q)_{n} = \prod_{i\in\BZ}\Hom_{R}(P_{i},Q_{i+n})\,.
  \]
This module is the $n$th component of a complex ${\Hom}_{R}(P,Q)$, with differential

\[
\partial(\beta)=\partial^Q\beta-(-1)^{|\beta|}\beta\partial^P.
\]

A map $\beta\col P\to Q$ with $\beta_{i}=0$ for $i\gg 0$ is called a \emph{bounded map}. The bounded maps form the subcomplex:
\[
\ov{\Hom}_{R}(P,Q)_{n} = \bigoplus_{i\in\BZ}\Hom_{R}(P_{i},Q_{i+n})\quad\text{for}\quad n\in\BZ\,.
\]

We write $\wh{\Hom}_{R}(P,Q)$ for the quotient complex.  It is proved in \cite{AV} that this complex is independent of the choices of $P$ and
$Q$ up to $R$-linear homotopy. By construction there is an exact sequence of DG $\End_R(Q)-\End_R(P)$-bimodules
\begin{equation}
\label{eq:bes}
0\lra \ov{\Hom}_{R}(P,Q)\lra\Hom_{R}(P,Q) \lra \wh{\Hom}_{R}(P,Q)\lra 0\,.
\end{equation}
The \emph{stable cohomology} of the pair $(L,M)$ is the graded $R$-module
$\Extv{}RLM$ with
\[
\Extv nRLM = \HH{n}{\wh{\Hom}_{R}(P,Q)}\quad\text{for each}\quad n\in\BZ\,.
\]
The \emph{bounded cohomology} of the pair $(L,M)$ is the graded $R$-module $\ov{\ext}_R(L,M)$ with
\[
\Extb nRLM = \HH{n}{\ov{\Hom}_{R}(P,Q)}\quad\text{for each}\quad n\in\BZ\,.
\]
The sequence \eqref{eq:bes} defines an exact sequence
\begin{equation}
\label{lesR}
\begin{gathered}
\xymatrixcolsep{2pc}
\xymatrixrowsep{0.5pc}
\xymatrix{
\ov\ext_R(L,M) \ar@{->}[r]^{\eta_R} &\ext_R(L,M) \ar@{->}[r]^{\iota_R} & \wh\ext_R(L,M) \ar@{->}[r]^{\quad\quad\;\;\eth_R} & \\
\Sigma\ov\ext_R(L,M) \ar@{->}[r]^{\Sigma\eta_R} &\Sigma\ext_R(L,M)}
\end{gathered}
\end{equation}
of graded $\ext_R(M,M)$-$\ext_R(L,L)$-bimodules.
We refer to \cite{AV} for a treatment on stable cohomology.
\section{A bimodule structure on the complex of bounded maps}
Let $(R,\fm,k)$ be a commutative local Noetherian ring. In this section we define a right $\ext_R(k,k)$-module structure on $\ext_R(k,R)\otimes_k\tor^R(k,N)$ that makes it isomorphic to $\Extb {}RkN$. This answers a question raised by Avramov and Veliche in \cite{AV}. For the rest of the paper $F$ will denote the acyclic closure of $k$, i.e.~a DG algebra minimial free resolution of $k$ with divided powers, see \cite[6.3]{infinite} for details. We denote by $\Der^\gamma_R(F)$ the subcomplex of $\End_R(F)$ of $\Gamma$-derivations, i.e.~$R$-linear endomorphisms of $F$ satisfying the Leibniz rule and respecting the divided power structure of $F$; see \cite[6.2.2]{infinite}. The complex $\Der^\gamma_R(F)$ is a DG Lie $R$-subalgebra of $\End_R(F)$, where the Lie structure on $\End_R(F)$ is defined as
\[
[\theta,\xi]:=\theta\xi-(-1)^{|\theta||\xi|}\xi\theta,\quad\mathrm{for\;}\theta,\xi\in\End_R(F)
\]
\[
\zeta^{[2]}:=\zeta^2,\quad\zeta\in\End_R(F)^{\mathrm{odd}}.
\]
Let $N$ be an $R$-module and $G$ a free resolution of $N$. We define a structure of DG $\End_R(G)$-$\Der^\gamma_R(F)$-bimodule on $\Hom_R(F,R)\otimes_R(F\otimes_RG)$. For $\alpha\in\End_R(G),\theta\in\Der^\gamma_R(F),\varphi\in\Hom_R(F,R),f\in F,g\in G$, we set the left and right products as follow:
\begin{align*}
&\alpha\cdot(\varphi\otimes f\otimes g):=(-1)^{|\alpha|(|\varphi|+|f|)}\varphi\otimes f\otimes\alpha(g),\\
&(\varphi\otimes f\otimes g)\cdot\theta:=(-1)^{|\theta|(|f|+|g|)}((\varphi\theta)\otimes f\otimes g-\varphi\otimes\theta(f)\otimes g).
\end{align*}
The right action is the tensor product action as defined in Section 1, with right action on $F\otimes_R G$ obtained by changing the canonical left action to a right action as explained in Section 1: for $\theta\in\Der^\gamma_RF, f\in F, g\in G$, the canonical left action is
\[
\theta\cdot(f\otimes g):=\theta(f)\otimes g.
\]
\begin{samepage}
\begin{proposition}
Let $\fg$ be a DG Lie $R$-algebra and $A$ a DG $R$-algebra with a structure of DG right $\fg$-module satisfying
\begin{equation}\label{der}
(ab)\theta=a(b\theta)+(-1)^{|\theta||b|}(a\theta)b,\quad\mathrm{for\;all\;}a,b\in A, \theta\in\fg.
\end{equation}
Let $M$ be a DG right $A$-module that is also a DG right $\fg$-module. Let $N$ be a DG left $A$-module that is also a DG right $\fg$-module. If for all $a\in A, m\in M, n\in N, \theta\in\fg$
\begin{equation}\label{compcond}
(an)\theta=(-1)^{|\theta||n|}(a\theta)n+a(n\theta),\quad\mathrm{and}\quad (ma)\theta=(-1)^{|a||\theta|}(m\theta)a+m(a\theta),
\end{equation}
then the DG right $\fg$-module structure of $M\otimes_RN$ induces a DG right $\fg$-module structure on $M\otimes_AN$.
\end{proposition}\end{samepage}
\begin{proof} Condition \eqref{der} is needed to ensure that if $a,b\in A$, $n\in N$ and $\theta\in\fg$ then $((ab)n)\theta=(a(bn))\theta$. In fact
\begin{align*}
((ab)n)\theta &=(-1)^{|\theta||n|}((ab)\theta)n+(ab)(n\theta)\\
              &=(-1)^{|\theta|(|b|+|n|)}((a\theta)b)n+(-1)^{|\theta||n|}(a(b\theta))n+(ab)(n\theta).
\end{align*}
The first equality comes from the first equation in \eqref{compcond} and the second equality follows from \eqref{der}. On the other hand
\begin{align*}
(a(bn))\theta &=(-1)^{|\theta|(|b|+|n|)}(a\theta)(bn)+a((bn)\theta)\\
              &=(-1)^{|\theta|(|b|+|n|)}(a\theta)(bn)+(-1)^{|\theta||n|}a((b\theta)n)+a(b(n\theta)).
\end{align*}
The first equality comes from the first equation in \eqref{compcond} and the second equality follows from \eqref{der}. The two expressions are the same since $N$ is a left $A$-module. Similarly one can prove $(m(ab))\theta=((ma)b)\theta$ with $m\in M$.

Recall that tensor product over $A$ is defined by the exactness of the sequence
\[
M\otimes_RA\otimes_RN\xra{\eta}M\otimes_RN\rightarrow M\otimes_AN\rightarrow0
\]
where $\eta$ is the map
\[
\eta: m\otimes a\otimes n\mapsto ma\otimes n-m\otimes an.
\]
It remains to prove that the image of $\eta$ is a DG right $\fg$-module. Indeed
\begin{align*}
\eta(m\otimes a\otimes n)\theta & = (ma\otimes n)\theta-(m\otimes an)\theta\\
&= (-1)^{|\theta||n|}(ma)\theta\otimes n+ma\otimes n\theta-(-1)^{|\theta|(|a|+|n|)}m\theta\otimes an-m\otimes(an)\theta\\
%&\quad\quad\quad -m\otimes(an)\theta\\
&=(-1)^{|\theta||n|+|a||\theta|}(m\theta)a\otimes n+(-1)^{|\theta||n|}m(a\theta)\otimes n+\\
&\quad\quad +ma\otimes n\theta-(-1)^{|\theta||a|+|\theta||n|}m\theta\otimes an+\\
&\quad\quad -(-1)^{|\theta||n|}m\otimes(a\theta)n-m\otimes a(n\theta)\\
&= (-1)^{|\theta||n|+|a||\theta|}(m\theta)a\otimes n-(-1)^{|\theta||a|+|\theta||n|}m\theta\otimes an+\\
&\quad\quad+ma\otimes n\theta-m\otimes a(n\theta)+\\
&\quad\quad+(-1)^{|\theta||n|}m(a\theta)\otimes n-(-1)^{|\theta||n|}m\otimes(a\theta)n.\qedhere
\end{align*}
\end{proof}
As a corollary we get
\begin{corollary}
The complex
\[
\Hom_R(F,R)\otimes_F(F\otimes_RG)
\]
is an $\End_R(G)$-$\Der^\gamma_R(F)$-bimodule with structure induced by the bimodule structure of $\Hom_R(F,R)\otimes_R(F\otimes_RG)$.\qed

\end{corollary}

We want to point out that $F$ is a DG left $\Der^\gamma(F)$-module by evaluation, so its right structure is
\[
f\cdot\theta:=-(-1)^{|f||\theta|}\theta\cdot f=-(-1)^{|f||\theta|}\theta(f)\quad\quad\quad f\in F, \theta\in\Der^\gamma_R(F).
\]

The $\End_R(G)$-$\Der_R^\gamma$(F)-bimodule structure of $\Hom_R(F,G)$ and $\ov\Hom_R(F,G)$ is given by left
and right composition.
\begin{theorem} \label{omegathm}
The following map is an isomorphism of DG $\End_R(G)$-$\Der^\gamma_R(F)$-bimodules:
\begin{equation}\label{omega}
\omega:\Hom_R(F,R)\otimes_F(F\otimes_R G)\rightarrow\ov\Hom_R(F,G)
\end{equation}
\[
\varphi\otimes x\otimes y\mapsto(f\mapsto(-1)^{|f||y|}\varphi(xf)y).
\]
\end{theorem}
\begin{proof}
The map $\omega$ is bijective since it is the compositions of the following maps
\[
\Hom_R(F,R)\otimes_F(F\otimes_RG)\xra{\cong}\Hom_R(F,R)\otimes_RG\xra{\cong}\ov\Hom_R(F,G).
\]
The first map is tensor cancellation and the second map is bijective by \cite[1.3.3]{AV}.
In the following $\alpha\in\mathrm{End}_R(G);\varphi\in\Hom_R(F,R);f,x\in F;y\in G;\theta\in\Der^\gamma_R(F)$.
We first check left linearity:
\begin{align*}
\omega(\alpha\cdot(\varphi\otimes x\otimes y))(f)&=\omega((-1)^{|\alpha|(|\varphi|+|x|)}\varphi\otimes x\otimes\alpha(y))(f)\\
&=(-1)^{|\alpha|(|\varphi|+|x|)+|f||\alpha(y)|}\varphi(xf)\alpha(y),
\end{align*}
and
\begin{align*}
(\alpha\cdot\omega(\varphi\otimes x\otimes y))(f)&=\alpha((-1)^{|f||y|}\varphi(xf)y)\\
&=(-1)^{|f||y|+|\alpha||\varphi(xf)|}\varphi(xf)\alpha(y).
\end{align*}
An elementary computation shows that the signs coincide.
For right linearity we observe
\begin{align*}
\omega((\varphi\otimes x\otimes y)\cdot\theta)(f)&=\omega((-1)^{|\theta|(|x|+|y|)}((\varphi\theta)\otimes x\otimes y-\varphi\otimes\theta(x)\otimes y))(f)\\
&=(-1)^{|\theta|(|x|+|y|)}(\omega((\varphi\theta)\otimes x\otimes y)(f)-\omega(\varphi\otimes\theta(x)\otimes y))\\
&=(-1)^{|\theta|(|x|+|y|)+|f||y|}(\varphi\theta(xf)y-\varphi(\theta(x)f)y)\\
&=(-1)^{|\theta|(|x|+|y|)+|f||y|}\varphi(\theta(xf)-\theta(x)f)y\\
&=(-1)^{|\theta|(|x|+|y|)+|f||y|+|\theta||x|}\varphi(x\theta(f))y,
\end{align*}
where the last equality holds because $\theta$ is a derivation. On the other hand
\begin{align*}
(\omega(\varphi\otimes x\otimes y)\cdot\theta)(f)&=\omega(\varphi\otimes x\otimes y)(\theta(f))\\
&=(-1)^{|\theta(f)||y|}\varphi(x\theta(f))y,
\end{align*}
an elementary computation shows that the signs coincide.
\end{proof}

Set $\pi(R) = \hh{\Der_R^\gamma(F)}$; it is a graded Lie $k$-algebra. By a structure theorem due to Milnor and Moore \cite{MilnorMoore} in characteristic 0 and to Andr\'{e} in characteristic $p>0$ \cite{Andre} (adjusted by Sj\"{o}din \cite{Sjodin2} for $p=2$), the universal enveloping algebra $U\pi(R)$ is isomorphic to $\ext_R(k,k)$. We recall that, by the discussion at the end of
Section 1, a right module over $\pi(R)$ is the same as a right module over $U\pi(R)$.

\begin{definition}
Let $A$ be a DG $R$-algebra, $C$ a right DG $A$-module and $D$ a left DG $A$-module. Then the following map is called a \emph{K\"{u}nneth map}
\[
\kappa: \hh{C}\otimes_{\hh{A}}\hh{D}\rightarrow\hh{C\otimes_AD}
\]
\[
[c]\otimes[d]\mapsto[c\otimes d],
\]
where $c\in C$ and $d\in D$.

\end{definition}

\begin{theorem}\label{main1}
Let $F$ be the acyclic closure of $k$ and $G$ a minimal free resolution of $N$. Then the K\"{u}nneth map
\[
\kappa:\hh{\Hom_R(F,R)}\otimes_{\hh{F}}\hh{F\otimes_RG}\rightarrow\hh{\Hom_R(F,R)\otimes_F(F\otimes_RG)}
\]
is an isomorphism of $\ext_R(N,N)$-$\ext_R(k,k)$-bimodules.
\end{theorem}
\begin{proof}
A straightforward computation shows the bilinearity of the map. We prove that it is bijective.
Let $I$ be a minimal injective resolution of $R$. Consider the following diagram:

\[
\begin{tikzpicture}
  \matrix (m) [matrix of math nodes,row sep=3em,column sep=4em,minimum width=2em] {
    \hh{\Hom_R(F,R)}\otimes_{\hh{F}}\hh{F\otimes_RG} & \hh{\Hom_R(F,R)\otimes_F(F\otimes_RG)}\\
    \hh{\Hom_R(F,I)}\otimes_{\hh{F}}\hh{F\otimes_RG} & \hh{\Hom_R(F,I)\otimes_F(F\otimes_RG)}\\
    \hh{\Hom_R(k,I)}\otimes_k\hh{F\otimes_RG} & \hh{\Hom_R(k,I)\otimes_F(F\otimes_RG)}\\
    & \hh{\Hom_R(k,I)\otimes_RG}\\
    \hh{\Hom_R(k,I)}\otimes_k\hh{k\otimes_RG} & \Hom_R(k,I)\otimes_RG\\
    \Hom_R(k,I)\otimes_kk\otimes_RG & \Hom_R(k,I)\otimes_kk\otimes_RG\\};
\path[->] (m-1-1) edge  node[above] {$\kappa$} (m-1-2);
\path[->] (m-1-1) edge  node[left] {$\varphi_1$} (m-2-1);
\path[->] (m-1-2) edge  node[right] {$\varphi_2$} (m-2-2);
\path[->] (m-2-1) edge  (m-2-2);
\path[->] (m-3-1) edge  node[left] {$\psi_1$} (m-2-1);
\path[->] (m-3-1) edge  (m-3-2);
\path[->] (m-3-2) edge  node[right] {$\psi_2$} (m-2-2);
\path[->] (m-3-1) edge  (m-5-1);
\path[->] (m-3-2) edge  (m-4-2);
%\draw[double,double distance=5pt] (m-5-1)  --   (m-6-1);
%\draw[double,double distance=5pt] (m-4-2)  --   (m-5-2);
%\draw[double,double distance=5pt] (m-5-2)  --   (m-6-2);
\path[->] (m-5-1) edge  node[left] {$=$} (m-6-1);
\path[->] (m-4-2) edge  node[right] {$=$} (m-5-2);
\path[->] (m-5-2) edge  node[right] {$=$} (m-6-2);
\path[->] (m-6-1) edge  node[above] {$id$} (m-6-2);
\end{tikzpicture}
\]
where $\varphi_1,\varphi_2$ are induced by the map $R\rightarrow I$ and $\psi_1,\psi_2$ are induced by the map $F\rightarrow k$. The map $\varphi_1$ is an isomorphism since the quasi-isomorphism $R\rightarrow I$ induces a quasi-isomorphism $\Hom_R(F,R)\rightarrow\Hom_R(F,I)$ because of the choice of $F$. Similarly for $\psi_1$.
The maps $\varphi_2,\psi_2$ are isomorphisms because $F\otimes_RG$ is a semifree $F$-module since it is a graded-free module bounded below over a non-negative DG algebra (see \cite{DGHA} for the definition of semifree DG module, see also \cite[Theorem 8.1]{DGHA}).
The bottom map is the identity. The complex $\mathrm{Hom}_R(k,I)\otimes_RG$ has zero differentials because $G$ is minimal and $\mathrm{Hom}_R(k,I)$ is a complex of $k$-vector spaces. The tensor product $\hh{\Hom_R(k,I)}\otimes_k\hh{F\otimes_RG}$ is isomorphic to $\hh{\Hom_R(k,I)}\otimes_k\hh{k\otimes_RG}$ since the tensor product is over a field and $\hh{F\otimes_RG}\cong\hh{k\otimes_RG}$. The last equality on the left follows by the minimality of $I$ and $G$.  The commutativity of the diagram follows by the naturality of the K\"{u}nneth map. This proves that all the horizontal maps are isomorphisms.
\end{proof}
\begin{corollary}\label{isobounded}
The composition
\[
\hh{\omega}\circ\kappa:\ext_R(k,R)\otimes_k\tor^R(k,N)\rightarrow\ov\ext_R(k,N)
\]
where $\kappa$ is the K\"{u}nneth map and $\omega$ is the map in \eqref{omega}, is an isomorphism of $\ext_R(N,N)$-$\ext_R(k,k)$-bimodules.
\end{corollary}

\section{The dual bimodule structure of $\tor^R(M,k)$}
As before, $F$ is an acyclic closure of $k$ and $G$ a free resolution of an $R$-module $N$. The complex $G\otimes_R F$ is a DG $\End_R G$-$\Der^\gamma_RF$-bimodule with actions
\[
\alpha\cdot(g\otimes f):=\alpha(g)\otimes f\quad\alpha\in\End_RG
\]
\[
(g\otimes f)\cdot\theta:=-(-1)^{|\theta||f|}g\otimes\theta(f)\quad\theta\in\Der^\gamma_RF
\]
where the right action is obtained by twisting the left action.\\
The complex $\Hom_F(G\otimes_RF,F)$ is a right DG $\End_RG$-module with product
\begin{equation}\label{eq:RightProd}
(\psi\cdot\alpha)(g\otimes f):=\psi(\alpha(g)\otimes f),
\end{equation}
with $\psi\in\Hom_F(G\otimes_RF,F),\;\alpha\in\End_RG,g\in G,f\in F$.
It has two structures of left DG $\Der^\gamma_RF$-module, let $\psi\in\Hom_F(G\otimes_RF,F),\theta\in\Der^\gamma_RF,f\in F,g\in G$ and define
\[
(\theta\ast_1\psi)(g\otimes f):=\theta(\psi(g\otimes f))
\]
\[
(\theta\ast_2\psi)(g\otimes f):=-(-1)^{|\theta|(|\psi|+|g|)}\psi(g\otimes\theta(f))
\]
where the second action is obtained by acting on the right on $g\otimes f$.
We combine these two left actions into a third one
\begin{equation}\label{thirdaction}
(\theta\cdot\psi)(g\otimes f):=\theta(\psi(g\otimes f))-(-1)^{|\theta|(|\psi|+|g|)}\psi(g\otimes\theta(f)).
\end{equation}
\begin{theorem} \label{dualaction}
The following map is an isomorphism of DG $\Der^\gamma_RF$-$\End_RG$-bimodules
\[
\chi:\Hom_R(G,F)\rightarrow\Hom_F(G\otimes_RF,F)
\]
\[
\varphi\mapsto(g\otimes f\mapsto\varphi(g)f)
\]
where $\Hom_R(G,F)$ has the canonical bimodule structure and $\Hom_F(G\otimes_RF,F)$ has the bimodule structure given by \eqref{eq:RightProd} and \eqref{thirdaction}.
\end{theorem}
\begin{proof}
The map $\chi$ is bijective because of the canonical isomorphism $F\rightarrow\End_FF$ and adjunction. We just need to check left and right linearity.\\
We start with right linearity, let $\alpha\in\End_RG,\varphi\in\Hom_R(G,F),f\in F,g\in G$
\[
\chi(\varphi\alpha)(g\otimes f)=\varphi\alpha(g)f
\]
\[
(\chi(\varphi)\alpha)(g\otimes f)=\chi(\varphi)(\alpha\cdot(g\otimes f))=\chi(\varphi)(\alpha(g)\otimes f)=\varphi\alpha(g) f,
\]
this proves $\chi(\varphi\alpha)=\chi(\varphi)\alpha$.\\
To prove left linearity let $\theta\in\Der^\gamma_RF$, then
\[
\chi(\theta\varphi)(g\otimes f)=\theta\varphi(g)f
\]
and
\begin{align*}
(\theta\chi(\varphi))(g\otimes f)&=\theta(\chi(\varphi)(g\otimes f))-(-1)^{|\theta|(|\varphi|+|g|)}\chi(\varphi)(g\otimes\theta(f))\\
&=\theta(\varphi(g)f)-(-1)^{|\theta|(|\varphi|+|g|)}\varphi(g)\theta(f)\\
&=\theta\varphi(g)f+(-1)^{|\theta|(|\varphi|+|g|)}\varphi(g)\theta(f)-(-1)^{|\theta|(|\varphi|+|g|)}\varphi(g)\theta(f)\\
&=\theta\varphi(g)f.
\end{align*}
where the third equality holds as $\theta$ is a derivation. This proves that $\chi(\theta\varphi)=\theta\chi(\varphi)$.
\end{proof}
In the following Corollary we consider $\ext_R(M,k)$ and $\Hom_k(\tor^R(M,k),k)$ as $\ext_R(k,k)$-$\ext_R(M,M)$-bimodules where the bimodule actions are induced in homology by the actions on $\Hom_R(G,F)$ and $\Hom_F(G\otimes_RF,F)$.
\begin{corollary}\label{DualOfTor}
The graded $\ext_R(k,k)$-$\ext_R(M,M)$-bimodules $\ext_R(M,k)$ and\\ \noindent $\Hom_k(\tor^R(M,k),k)$ are isomorphic.
\end{corollary}
\begin{proof}
We only need to prove that
\[
\hh{\Hom_F(G\otimes_RF,F)}\cong\Hom_k(\tor^R(M,k),k).
\]
Since $G\otimes_RF$ is a semifree DG $F$-module, the functor $\Hom_F(G\otimes_RF,-)$ preserves quasi-isomorphisms, hence
\[
\hh{\Hom_F(G\otimes_RF,F)}\cong\hh{\Hom_F(G\otimes_RF,k)}
\]
now it remains to notice that $\Hom_F(G\otimes_RF,k)=\Hom_k(G\otimes_RF,k)$ and
\[
\hh{\Hom_k(G\otimes_RF,k)}=\Hom_k(\hh{G\otimes_RF},k)
\]
since $k$ is a field.
\end{proof}

Let $R^!$ be the Koszul dual of $R$, i.e. the subalgebra of $\ext_R(k,k)$ generated by $\ext^1_R(k,k)$. In \cite[3.3]{hi} it is proved that $\ext_R(M,k)$ and $\Hom_k(\tor^R(M,k),k)$ are isomorphic as right $R^!$-modules where the action on $\ext_R(M,k)$ is the usual left action twisted with the antipode map. We want to prove that the left $\ext_R(k,k)$-action of the previous corollary restricts to the left $R^!$-action of \cite[3.3]{hi} once we turn the modules from right to left using the antipode map.
\begin{proposition}
The left $\ext_R(k,k)$-action on $\Hom_k(\tor^R(M,k),k)$ as defined in \ref{thirdaction} restricts to the left $R^!$-action as defined in \cite[3.3]{hi}.
\end{proposition}
\begin{proof}
Let $e$ be the embedding dimension of $R$ and $y_1,\ldots,y_e$ an algebra basis of $R^!$ as constructed in \cite[2.9]{hi}. We can choose the previous basis so that $y_i=[\theta_i]$ with $\theta_i\in\Der^\gamma_R(F)$, see \cite[10.2.1]{infinite}. We want to understand the action of $y_i$. To make the proof more readable we drop the subscript $i$ for $y_i$ and $\theta_i$.

Let $G$ be a projective resolution of $M$ and $F$ an acyclic closure of $k$. Consider the following commutative diagram where $\varepsilon: F\rightarrow k$ is the augmentation
\begin{equation} \label{commdiag}
\begin{tikzpicture}
  \matrix (m) [matrix of math nodes,row sep=3em,column sep=4em,minimum width=2em] {
   \Hom_R(G,F) & \Hom_F(G\otimes_RF,F)\\
   \Hom_R(G,k) & \Hom_F(G\otimes_RF,k).\\};
\path[->] (m-1-1) edge  node[above]{$\chi$}(m-1-2);
\path[->] (m-2-1) edge  (m-2-2);
\path[->] (m-1-1) edge  node[left]{$\Hom_R(G,\varepsilon)$} (m-2-1);
\path[->] (m-1-2) edge  node[right]{$\Hom_F(G\otimes_RF,\varepsilon)$} (m-2-2);
\end{tikzpicture}
\end{equation}
The vertical maps are quasi-isomorphisms. Take $\alpha\in\Hom_R(F,k)$ of degree $1-n$ and $z\in G$ of degree $n$. Denote by $\bar z$ the image of $z$ in $G\otimes_R k=\tor^R(M,k)$. Denote by $\varphi$ the element $\chi(\tilde \alpha)$ where $\tilde\alpha$ is a lifting of $\alpha$ to $\Hom_R(G,F)$. By Theorem \ref{dualaction} we have
\[
(\theta\varphi)(g\otimes f)=\theta(\tilde\alpha(g))f.
\]
Let $r_1,\ldots, r_e$ be a minimal generating set of $\fm$, and let $\partial(z)=\sum_{j=1}^er_jf_j$ with $f_j\in F_{n-1}$. Now we calculate $[\theta\varphi](\bar z)$ by lifting $\bar z$ to $G\otimes_RF$
\begin{align*}
[\theta\varphi](\bar z) &=\varepsilon((\theta\varphi)(z\otimes1+\cdots))\\
&= \varepsilon\theta(\tilde\alpha(z))\\
&=(-1)^{|\alpha|}\varepsilon\tilde\alpha(f_i)\\
&=(-1)^{|\alpha|}\alpha(f_i)
\end{align*}
where the first equality follows from the commutativity of the diagram \eqref{commdiag}, the second from \ref{dualaction} and degree reasons, the third from \cite[2.12]{hi}, the fourth from the definition of $\tilde\alpha$.

In the proof of \cite[3.3]{hi} it is proved that  $(\alpha\cdot y_i)(z)=-\alpha(f_i)$, twisting this into a left action using $y_i\cdot\alpha:=-(-1)^{|\alpha|}\alpha\cdot y_i$ we get that this left action is the same as the one defined in \ref{thirdaction}.
\end{proof}
\section{Gorenstein rings}
The stable cohomology of a pair of modules is zero if $R$ is regular since every module admits a finite free resolution. From now on we will assume that $R$ is a singular ring (i.e. not regular). If $R$ is Gorenstein, then $\ext_R(k,R)\cong\Sigma^{-d}k$ with $d=\dim R$ where $\Sigma$ is the suspension functor, hence by Corollary \ref{isobounded}
\[
\ov\ext_R(k,N)\cong\Sigma^{-d}\tor^R(k,N)
\]
as $\ext_R(N,N)$-$\ext_R(k,k)$-bimodules. This is because $\Sigma^{-d}k\otimes_k\tor^R(k,N)$ is isomorphic to $\Sigma^{-d}\tor^R(k,N)$ as $\ext_R(N,N)$-$\ext_R(k,k)$-bimodules.

We will use the following notation
\[
\mathcal{S}=\wh\ext_R(k,k),\quad\mathcal{E}=\ext_R(k,k),\quad\mathcal{B}=\ov\ext_R(k,k).
\]
In \cite[5.1.8]{AV} (and \cite[Theorem 6]{mart}) it is proved that the map $\eta_R$ for the pair $(k,k)$ in the sequence \eqref{lesR} is zero, yielding an exact sequence of $\mathcal{E}$-bimodules
\begin{equation} \label{mart}
0\rightarrow\mathcal{E}\xra{\iota}\mathcal{S}\rightarrow\Sigma\mathcal{B}\rightarrow0.
\end{equation}
\begin{definition}
Let $k$ be a field and $\mathcal{A}$ a graded $k$-algebra with $\ma^0=k$ and $\ma^i=0$ for all $i<0$. Let $\mm$ be a graded left $\ma$-module. Set
\[
\Gamma\mm=\bigcup_{i=0}^\infty\{\mu\in\mm\mid\ma^{\geq i}\mu=0\}.
\]
The left torsion $\me$-subbimodule of $\ms$ is
\[
\mt:=\Gamma\ms.
\]
\end{definition}

\begin{lemma} \label{split}
If $\mathcal{S}=\iota(\me)\oplus\mt^\prime$ for some graded $\me$-subbimodule $\mt^\prime$ of $\ms$, then $\mt^\prime=\mt$ and
\[
\mt^\prime\cong\Sigma^{1-d}\tor^R(k,k)
\]
as graded $\me$-bimodules.
\end{lemma}
\begin{proof}
Our hypothesis and \eqref{mart} implies that $\mt^\prime$ is isomorphic to $\Sigma\mb$ as graded $\me$-bimodules. By \cite[(7.3.2)]{AV} $\mb=\Gamma\mb$, hence the following containments hold
\[
\mt^\prime=\Gamma\mt^\prime\subseteq\Gamma\ms=\mt.
\]
By \cite[(7.3.4)]{AV} one has $\iota(\me)\cap\mt=(0)$, and since $\ms=\iota(\me)\oplus\mt^\prime$ we deduce $\mt\subseteq\mt^\prime$; this gives us
\[
\mt=\mt^\prime. \qedhere
\]
\end{proof}

\begin{definition}
Let $k$ be a field and $\ma$ a graded $k$-algebra with $\ma^0=k$ and $\ma^i=0$ for all $i<0$. Let $\mm$ be a graded left $\ma$-module. The \textit{depth} of $\mm$ over $\ma$ is defined as
\[
\depth_\ma\mm=\inf\{n\in\mathbb{N}\mid\ext^n_\ma(k,\mm)\neq0\}.
\]
\end{definition}

\begin{definition}
Let $A$ be a graded $k$-algebra with $k$ a field. Let $M$ be a graded $A$-bimodule. The trivial extension algebra of $A$ by $M$ is an algebra denoted by $A\ltimes M$ with underlying bimodule $A\oplus M$ and product given by
\[
(a,m)\cdot(b,n):=(ab,an+mb).
\]
\end{definition}

\begin{theorem} \label{trivialext}
If $R$ is a Gorenstein ring with $\depth\me\geq2$, then the stable cohomology algebra is a trivial extension algebra,
\[
\ms\cong\me\ltimes\Sigma^{1-d}\me^\vee.
\]
\end{theorem}

\begin{proof}
By a dual version of Corollary \ref{DualOfTor} it is enough to prove that $\ms\cong\me\ltimes\Sigma^{1-d}\tor^R(k,k)$. By \cite[7.2(3)]{AV} if $\depth\me\geq2$ then $\ms=\iota(\me)\oplus\mt$ as $\me$-bimodules, hence by Lemma \ref{split} $\mt\cong\Sigma^{1-d}\tor^R(k,k)$ as graded $\me$-bimodules. By \cite[9.2(3)]{AV} $\mt\cdot\mt=0$, hence $\ms$ is a trivial extension of $\iota(\me)$ and $\mt$.
\end{proof}
\begin{remark}
If $R$ is a complete intersection then by \cite[8.3]{AV} $\depth\me=\codim R$, hence any complete intersection with $\codim R\geq 2 $ satisfies the hypothesis of Theorem \ref{trivialext}. The structure of $\ms$ for hypersurfaces is already known, see \cite[(10.2.3)]{buch} (see also \cite[8.4]{AV}). In this case, stable cohomology $\wh\ext_R(k,k)$ is a central localization of absolute cohomology $\ext_R(k,k)$.
\end{remark}
\section{Commutativity}
We recall that by \cite{sjodin} the algebra $\ext_R(k,k)$ is
graded-commutative if and only if $\wh R=Q/I$ with $(Q,\fn)$ regular ring, $I$ generated by a regular sequence, and $I\subseteq\fn^3$.\\
In the following $F$ is the acyclic closure of $k$ over $R$ and $G$ a $R$-free resolution of $N$. We compute $\ext_R(N,N)$ using the complex $\mathrm{End}_R(G)$. We compute $\tor^R(k,N)$ using the complex $F\otimes_R G$. The $\ext_R(N,N)$-$\ext_R(k,k)$-bimodule structure of $\tor^R(k,N)$ is defined as before, i.e. let $[\alpha]\in\ext_R(N,N)$ and $[\Sigma_if_i\otimes g_i]\in\tor^R(k,N)$ then
\[
[\alpha]\cdot[\sum_if_i\otimes g_i]=[\sum_i(-1)^{|\alpha||f_i|}f_i\otimes\alpha(g_i)],
\]
and for $[\theta]\in\pi(R)$
\[
[\sum_if_i\otimes g_i]\cdot[\theta]=-[\sum_i(-1)^{|\theta|(|f_i|+|g_i|)}\theta(f_i)\otimes g_i].
\]
The following lemma is well-known:
\begin{lemma}\label{completion}
Let $(R,\fm,k)\rightarrow (R',\fm',k')$ be a local homomorphism such that the $R$-module $R'$ is flat and $R'\otimes_Rk\cong k'$. Let $N$ be a finitely generated $R$-module and let $N'$ be the $R'$-module $R'\otimes_RN$.
There are isomorphisms of algebras
\begin{align*}
&\alpha:R'\otimes_R\ext_R(k,k)\rightarrow\ext_{R'}(k',k')\\
&\beta:R'\otimes_R\ext_R(N,N)\rightarrow\ext_{R'}(N',N').
\end{align*}
The canonical map $\varphi:R'\otimes_R\tor^R(k,N)\rightarrow\tor^{R'}(k',N')$ is bijective and $\beta$-$\alpha$-covariant.\qed
\end{lemma}
\begin{definition}
Let $A$ be a graded $k$-algebra and $M$ a graded $A$-bimodule. We say that $M$ is symmetric if for every $m\in M$ and $a\in A$
\[
am=(-1)^{|a||m|}ma.
\]
\end{definition}

\begin{theorem} \label{symmetry}
If $R$ is a complete intersection $\wh R=Q/I$ with $I$ generated by a regular sequence, $(Q,\fn)$ regular ring, $I\subseteq\fn^3$, then $\tor^R(k, k)$ is a symmetric $\ext_R(k, k)$-bimodule.
\end{theorem}
\begin{proof}
We first investigate how $\pi(R)$ acts on $\tor^R(k, k)$.
By Lemma \ref{completion} we can assume that $R$ is complete and $R=Q/I$ with $\fn=( a_1,\ldots, a_e)$, $I=(f_1,\ldots,f_c)$, $Q$ regular, $f_1,\ldots,f_c$ a $Q$-sequence and $I\subseteq \fn^2$. Write
\[
f_i=\sum_{j\leq k}r_{ijk} a_ja_k.
\]
Let $F$ be the acyclic closure of $k$ over $R$. If $a\in Q$ we denote by $\bar a$ the class of $a$ in $R$. Then by \cite{tate}
\[
F=R\langle x_1,\ldots,x_e,y_1,\ldots,y_c\mid \partial(x_i)=\bar a_i,\quad\partial(y_i)=\sum_{j\leq k}\bar r_{ijk}\bar a_jx_k\rangle
\]
is the acyclic closure of $k$ over $R$.
Since $R$ is a complete intersection, by \cite{sjodin} $\pi(R)$ is generated as a $k$-vector space by
elements $\xi_1,\ldots,\xi_e$ of degree 1 and elements $\chi_1,\ldots,\chi_c$ of degree 2, where $e$ is the
embedding dimension of $R$ and $c$ its codimension. These generators are classes of derivations
of $F$ defined as follows
\begin{alignat*}{3}
&\xi_t(x_i)=\delta_{it} & \quad\mathrm{and}\quad & \xi_t(y_i)=-\sum_{j\leq t}\bar r_{ijt} x_j\\
&\chi_t(x_i)=0 & \quad\mathrm{and}\quad & \chi_t(y_i)=\delta_{it}
\end{alignat*}
The generators (as an algebra) of degree 1 of $\tor^R(k,k)$, which is $\mathrm{H}(F\otimes_RF)$, are the classes of
\[
x_i\otimes1-1\otimes x_i\quad\quad\quad i=1,\ldots,e
\]
and the generators of degree 2 are the classes of
\[
y_i\otimes1+\sum_{j\leq k}\bar r_{ijk}x_k\otimes x_j-\sum_{j\leq k}\bar r_{ijk}\otimes x_kx_j-1\otimes y_i
\]
for $i=1,\ldots,c$.

Now we check how the derivations act on these generators
\[
\xi_j\cdot(x_i\otimes1-1\otimes x_i)=-1\otimes\delta_{ij},
\]
\[
(x_i\otimes 1-1\otimes x_i)\cdot\xi_j=\delta_{ij}\otimes1,
\]
\[
\chi_t\cdot(y_i\otimes1+\sum_{j\leq k}\bar r_{ijk}x_k\otimes x_j-\sum_{j\leq k}\bar r_{ijk}\otimes x_kx_j-1\otimes y_i)=-1\otimes\delta_{ti},
\]
\[
(y_i\otimes1+\sum_{j\leq k}\bar r_{ijk}x_k\otimes x_j-\sum_{j\leq k}\bar r_{ijk}\otimes x_kx_j-1\otimes y_i)\cdot\chi_t=-\delta_{ti}\otimes1,
\]
\begin{align*}
&\xi_t\cdot(y_i\otimes1+\sum_{j\leq k}\bar r_{ijk}x_k\otimes x_j-\sum_{j\leq k}\bar r_{ijk}\otimes x_kx_j-1\otimes y_i)=\\&\quad\quad-\sum_{j\leq t}\bar r_{ijt}x_t\otimes 1-\sum_{j\leq t} \bar r_{ijt}\otimes x_j+\sum_{t\leq k}\bar r_{itk}\otimes x_k+\sum_{j\leq t}\bar r_{ijt}\otimes x_j,
\end{align*}
\[
(y_i\otimes1+\sum_{j\leq k}\bar r_{ijk}x_k\otimes x_j-\sum_{j\leq k}\bar r_{ijk}\otimes x_kx_j-1\otimes y_i)\cdot\xi_t=\sum_{j\leq t}\bar r_{ijt}x_j\otimes 1-\sum_{j\leq t}\bar r_{ijt}\otimes x_t.
\]
So
\[
\xi_j\cdot(x_i\otimes 1-1\otimes x_i)=-(x_i\otimes1-1\otimes x_i)\xi_j,
\]
\begin{align*}
&\chi_t\cdot(y_i\otimes1+\sum_{j\leq k}\bar r_{ijk}x_k\otimes x_j-\sum_{j\leq k}\bar r_{ijk}\otimes x_kx_j-1\otimes y_i)=\\&\quad\quad(y_i\otimes1+\sum_{j\leq k}\bar r_{ijk}x_k\otimes x_j-\sum_{j\leq k}\bar r_{ijk}\otimes x_kx_j-1\otimes y_i)\cdot\chi_t
\end{align*}
therefore only the action of elements of degree 1 on elements of degree 2 might break the
symmetry. The $r_{ijk}$ are in $\fm$ since $I\subseteq\fn^3$, hence applying $\varepsilon\otimes F$ (where $\varepsilon:F\rightarrow k$ is the augmentation) to
\begin{align*}
&\xi_t\cdot(y_i\otimes1+\sum_{j\leq k}\bar r_{ijk}x_k\otimes x_j-\sum_{j\leq k}\bar r_{ijk}\otimes x_kx_j-1\otimes y_i)=\\&\quad\quad-\sum_{j\leq t}\bar r_{ijt}x_t\otimes 1-\sum_{j\leq t} \bar r_{ijt}\otimes x_j+\sum_{t\leq k}\bar r_{itk}\otimes x_k+\sum_{j\leq t}\bar r_{ijt}\otimes x_j,
\end{align*}
and
\[
(y_i\otimes1+\sum_{j\leq k}\bar r_{ijk}x_k\otimes x_j-\sum_{j\leq k}\bar r_{ijk}\otimes x_kx_j-1\otimes y_i)\cdot\xi_t=\sum_{j\leq t}\bar r_{ijt}x_j\otimes 1-\sum_{j\leq t}\bar r_{ijt}\otimes x_t.
\]
yields zero, i.e. the left and right product of the class of a derivation of degree 1 on a cycle of degree 2 is zero in homology. We just proved that the action of $\pi(R)$ is symmetric, but $\ext_R(k,k)\cong U\pi(R)$ and by \cite{sjodin} $\ext_R(k,k)$ is graded-commutative, hence the action of $\ext_R(k,k)$ is symmetric.
\end{proof}
We show that $\tor^R(k,k)$ is not in general symmetric.
\begin{example}
Let $R=k[\![x,y]\!]/(xy)$, and denote by $\bar x, \bar y$ the classes of $x$ and $y$ in $R$. The acyclic closure of $k$ over $R$ is
\[
F=R\langle T_1,T_2,S\mid \partial T_1=\bar x, \partial T_2=\bar y, \partial S=\bar x T_2\rangle.
\]
Consider the cycle
\[
z=S\otimes 1+T_2\otimes T_2-1\otimes T_2T_1-1\otimes S
\]
and the derivation defined by
\[
\xi(T_1)=0,\quad\xi(T_2)=1,\quad\xi(S)=-T_1
\]
then
\begin{align*}
\xi\cdot z&=-1\otimes\xi(T_2T_1)-1\otimes\xi(S)\\
&=-1\otimes(\xi(T_2)T_1-T_2\xi(T_1))+1\otimes T_1\\
&=-1\otimes\xi(T_2)T_1+1\otimes T_1\\
&=0
\end{align*}
but
\begin{align*}
z\cdot\xi &= -\xi(S)\otimes1-\xi(T_2)\otimes T_1\\
&=-T_1\otimes 1-1\otimes T_1\\
&=-(T_1\otimes 1+1\otimes T_1)
\end{align*}
This shows that the action is not symmetric.
\end{example}
Denote by $\pi^1$ and $\fg^1$ $k$-vector spaces of rank $e$ and $\pi^2,\fg^2$ $k$-vector spaces of rank $c$. Fix bases
\[
\pi^1=\langle \xi_1,\ldots, \xi_e\rangle\quad\pi^2=\langle \chi_1,\ldots, \chi_c\rangle
\]
\[
\fg^1=\langle x_1,\ldots, x_e\rangle\quad\fg^2=\langle y_1,\ldots, y_c\rangle.
\]
Denote by $E$ the graded $k$-algebra $E=\bigwedge\pi^1\otimes_k\mathrm{Sym}(\pi^2)$ and by $T$ the $k$-vector space $T=\bigwedge\fg^1\otimes_k\mathrm{Sym}(\fg^2)$, we give $T$ the following graded $E$-bimodule structure
\[
\xi_jx_i=-\delta_{ij},\quad x_i\xi_j=\delta_{ij},
\]
\[
\chi_jy_i=-\delta_{ji},\quad y_i\chi_j=-\delta_{ji},
\]
\[
\xi_j y_i=0,\quad y_i\xi_j=0.
\]
With this notation we prove
\begin{theorem}
The following conditions on a local ring $R$ are equivalent:\\
1) $\wh R\cong Q/I$ with $(Q,\fn)$ regular, $I$ generated by a regular sequence and $I\subseteq \fn^3$,\\
2) the $k$-algebra $\wh\ext_R(k,k)$ is graded-commutative.\\
Let $e$ and $c$ be the embedding dimension and codimension of $R$ respectively. When 1 or 2 hold $\wh\ext_R(k,k)\cong E\ltimes\Sigma^{1-d} T$ if $\codim R\geq2$. If $R$ is an hypersurface then $\wh\ext_R(k,k)\cong\bigwedge \pi^1\otimes_k k[t,t^{-1}]$ with $\deg t=2$.
\end{theorem}
\begin{proof}
$2)\Rightarrow 1)$ If $\wh\ext_R(k,k)$ is graded-commutative then so is $\ext_R(k,k)$ since it is a subalgebra (see \cite{mart}), and by \cite{sjodin} the ring $R$ has the desired form.\\
$1)\Rightarrow 2)$ if $\codim R\geq2$, then by \ref{trivialext}
\[
\wh\ext_R(k,k)\cong\ext_R(k,k)\ltimes\Sigma^{1-d}\tor^R(k,k)
\]
but by \cite{sjodin} if $R$ has the required form then $\ext_R(k,k)\cong E$, and by \ref{symmetry}
\[
\tor^R(k,k)\cong T.
\]
If the ring is an hypersurface, then by \cite{buch} $\wh\ext_R(k,k)$ has the desired form.\\
The algebras $E\ltimes T$ and $\bigwedge \pi^1\otimes_k k[t,t^{-1}]$ are clearly graded-commutative.\qedhere
\end{proof}

\section{Rings with $\fm^2=0$}
In this section we give a multiplication table for the right $\ext_R(k,k)$-module structure of $\ov\ext_R(k,k)\cong\ext_R(k,k)\otimes_k\tor^R(k,k)$ for local Noetherian rings of embedding dimension $e$ with $\fm^2=0$ (not necessarily Gorenstein). We start by noticing that for rings with $\fm^2=0$ the algebra $\ext_R(k,k)$ is a tensor algebra over $\ext_R^1(k,k)$ which we assume being generated by elements $y_1,\ldots,y_e$. It suffices to describe the action of the $y_i$'s.

First we describe the right module structure of $\ext_R(k,R)$. The exact sequence
\[
0\rightarrow \fm\rightarrow R\rightarrow k\rightarrow0
\]
induces an exact sequence of right $\ext_R(k,k)$-modules
\begin{equation} \label{presentation}
0\rightarrow\Sigma\ext_R(k,k)\xra{\eth}\ext_R(k,k)^e\rightarrow\ext_R(k,R)\rightarrow0
\end{equation}
since $\fm\cong k^e$. Computing the connecting homomorphism shows that
\[
\eth=\begin{pmatrix} y_1 \\
\vdots\\
y_e
\end{pmatrix}
\]
in particular \ref{presentation} is a minimal free resolution of $\ext_R(k,R)$, showing that this module is minimally generated by $e$ elements. Considering this module as a quotient of $\ext_R(k,k)^e$ we can denote its minimal generators as
\[
\begin{pmatrix} \bar1 \\
0\\
\vdots\\
0
\end{pmatrix}=v_1,
\cdots,
\begin{pmatrix} 0\\
\vdots \\
0\\
\bar1
\end{pmatrix}=v_e.
\]
Now we want to describe a minimal free resolution of $k$ over $R$. If $\fm=(x_1,\ldots,x_e)$ then we denote by $\delta: R^e\rightarrow R$ the composition of the map $R^e\rightarrow\fm$ given by $e_i\mapsto x_i$, where the $e_i$'s form a basis of $R^e$, with the inclusion $\fm\rightarrow R$. We denote $R^e$ by $U$. We denote by $F_{i+1}$ the module $U^{\otimes(i+1)}$, i.e. the $i+1$th tensor power of $U$. The map $\partial$ is the map
\[
F_{i+1}=U^{\otimes(i+1)}=U\otimes_RU^{\otimes i}\xra{\delta\otimes_R U^{\otimes i}} R\otimes U^{\otimes i}=U^{\otimes i}=F_i.
\]
The complex $(F,\partial)$ is a minimal free resolution of $k$, see \cite[Lemma 10.5]{AV}.

Consider an element $z\in F_{w}$, without loss of generality we can assume $z=u_1\otimes\cdots \otimes u_{w}$ with $u_1,\ldots,u_w\in U$. We denote by $\bar z$ the corresponding element in $k\otimes_RF=\tor^R(k,k)$. By \cite[2.11]{hi} and the last proof of the previous section we know that if $\partial(z)=\Sigma x_jf_j$ then
\[
\bar z\cdot y_i=-(-1)^{|z|}\bar f_i.
\]
By construction $z=u_1\otimes\cdots\otimes u_w$ and we can write $u_1$ as $\Sigma r_ie_i$. By the definition of $\partial$ we get
\[
\partial(z)=\delta(u_1)\otimes u_2\otimes\cdots u_w=\Sigma r_ix_i\otimes u_2\otimes\cdots\otimes u_w.
\]
It follows that
\[
\bar z\cdot y_i=-(-1)^{|z|}\ov{r_i\otimes u_2\otimes\cdots\otimes u_w}.
\]
And finally if $v_j$ is one of the minimal generators of $\ext_R(k,R)$ and $ z$ is an homogeneous element in $F$ we get
\[
(v_j\otimes\bar z)\cdot y_i=(-1)^{|z|}v_jy_i\otimes1-(-1)^{|z|}v_j\otimes\ov{r_i\otimes u_2\otimes\cdots\otimes u_w}.
\]

\section*{Acknowledgments}
The author thanks his advisors Luchezar Avramov and Srikanth Iyengar for their help in completing this project and Frank Moore for reading an early version of this paper.

\end{document}